\documentclass[11pt]{amsart}

\usepackage{amssymb}
\usepackage{amsthm}
\usepackage{amsmath} 
\usepackage{amsbsy}
\usepackage{hyperref}
\usepackage{tikz}
\usepackage{enumerate}
\usepackage[margin=1.4in]{geometry}

\newtheorem{theorem}{Theorem}

\newtheorem{proposition}[theorem]{Proposition}

\newtheorem*{question}{Question}

\begin{document}

\title[]{The smoothest average: \\Dirichlet, Fej\'er and Chebyshev} 


\author[]{Noah Kravitz}
\address{Grace Hopper College, Zoom University at Yale, New Haven, CT 06511}
\email{noah.kravitz@yale.edu}

\author[]{Stefan Steinerberger}
\address{Department of Mathematics, University of Washington, Seattle, WA 98195}
\email{steinerb@uw.edu}

\begin{abstract} We are interested in the ``smoothest'' averaging that can be achieved by convolving functions $f \in \ell^2(\mathbb{Z})$  with an averaging function $u$.  More precisely, suppose $u:\{-n, \ldots, n\} \to \mathbb{R}$ is a symmetric function normalized to $\sum_{k=-n}^{n}u(k) = 1$. We show that every convolution operator is not-too-smooth, in the sense that
$$\sup_{f \in \ell^2(\mathbb{Z})} \frac{\| \nabla (f*u)\|_{\ell^2(\mathbb{Z})}}{\|f\|_{\ell^2}}\geq \frac{2}{2n+1},$$
and we show that equality holds if and only if $u$ is constant on the interval $\{-n, \ldots, n\}$.  In the setting where smoothness is measured by the $\ell^2$-norm of the discrete second derivative and we further restrict our attention to functions $u$ with nonnegative Fourier transform, we establish the inequality
$$\sup_{f \in \ell^2(\mathbb{Z})} \frac{\| \Delta (f*u)\|_{\ell^2(\mathbb{Z})}}{\|f\|_{\ell^2(\mathbb{Z})}} \geq \frac{4}{(n+1)^2},$$
with equality if and only if $u$ is the triangle function $u(k)=(n+1-|k|)/(n+1)^2$. We also discuss a continuous analogue and several open problems.
\end{abstract}
\maketitle

\section{Introduction and results}
\subsection{Smooth averaging} It is often desirable to obtain ``smoothed'' local averages of a function $f$ from either $\mathbb{Z}$ or $\mathbb{R}$ to $\mathbb{C}$.  Think, for instance, of the performance of an athlete over a certain period of time, or the value of a stock during a three-hour window in the afternoon: erratic local behavior (noise) in the raw data can conceal longer-term trends.  One natural way to reduce noise is to replace each value of $f$ with a weighted average of the nearby values of $f$.
Inspired by the axiomatic approach in game theory (see, e.g., \cite{har, nash, shapley}), one can ask if there is a particularly canonical way of obtaining such smooth averages.  Certainly, it is natural to require that the averaging process be invariant under translation (so that the process behaves the same everywhere). We should also like the averaging process to preserve overall size ($\ell^1$-mass in the discrete setting and $L^1$-mass in the continuous setting). These considerations suggest that we should average $f$ by convolving it with a (fixed) symmetric function $u$ that has the normalization $\sum_{k \in \mathbb{Z}} u(k)= 1$ or $\int_{\mathbb{R}} u(x) dx = 1$.  Moreover, we may wish to restrict the scale on which the averaging is done by either fixing a higher moment of $u$ or bounding its support.
But which function $u$ should one choose? The characteristic function on the desired length scale is certainly a classical choice, as is a Gaussian. In a certain sense, there is no single ``right'' answer---it is a matter of taste.  But there can be a right answer if we list additional desirable properties of $u$.
\begin{question}\label{q:main}
Which properties of an averaging kernel $u$ (and the resulting averages $f*u$) are desirable?  Which functions $u$ best satisfy these properties? 
\end{question}
\noindent This question has received significant attention in the context of image processing (e.g., \cite{babaud, hogan, linde, linde2, steini, vickrey, yu}).  The second author \cite{steini2} recently proposed a particular approach for the continuous setting: we can measure the smoothness of a function by the $L^2$-norm of its derivative and then ask which $u$ best uniformly minimizes the quantity $\| \nabla (f*u)\|_{L^2}$ (relative to $\|f\|_{L^2}$).  The investigation of this question led to new uncertainty principles for the Fourier transform, families of conjectured optimal averaging kernels, and interesting sign patterns in the hypergeometric function $_2F_1$.\\
The main purpose of the present paper is to consider analogous discrete problems.  Our results can thus be interpreted in the context of sharp uncertainty-type principle in harmonic analysis (in a similar flavor as, e.g., \cite{amrein, babenko, beckner, beckner2, bene, bene2, bened, bourgain, cohn, cow, folland,gon3, hirsch}).  They are also related to recent advances on other convolution-type inequalities in the discrete setting (e.g., \cite{car, ion, mag, stein}).

\subsection{Two sharp inequalities.}
We now discuss two inequalities that provide answers to our Question for certain natural notions of smoothness. Throughout this subsection, take $u:\left\{-n, \dots, n\right\} \to \mathbb{R}$ to be a symmetric (even) function with the normalization $$ \sum_{k=-n}^n {u(k)} = 1,$$
and suppose that we smooth a function $f \in \ell^2(\mathbb{Z})$ by convolving it with $u$ (and using $f * u$ as the smoothed average). 
First, suppose we measure smoothness of a function by the $\ell^2$-norm of its discrete derivative, which is given by $(\nabla f)(k)=f(k+1)-f(k)$.  So we wish to find $u$ that uniformly minimizes $\| \nabla (f*u) \|_{\ell^2}$ over all choices of $f$ with fixed $\|f\|_{\ell^2}$.  This problem turns out to have a particularly clean solution.

\begin{theorem}\label{thm:first-deriv}
Let $u:\left\{-n, \dots, n\right\} \to \mathbb{R}$ be a symmetric function with normalization $\sum_{k=-n}^n{u(k)}=1$. Then we have the inequality
$$\sup_{0 \neq f \in \ell^2(\mathbb{Z})}  \frac{\| \nabla (f*u)\|_{\ell^2(\mathbb{Z})}}{\|f\|_{\ell^2}} \geq \frac{2}{2n+1},$$
with equality if and only if $u$ is the constant function $u(k)=1/(2n+1)$.
\end{theorem}
\noindent This sharp inequality demonstrates a sense in which averaging in the classical way (i.e., with an unweighted mean of nearby values) is a reasonable strategy: it minimizes the worst-case first-order oscillation in the smoothed function.  Theorem~\ref{thm:first-deriv} also lends support to a conjecture in \cite{steini2} that the characteristic function is extremal for a related continuous problem. 
Suppose we measure smoothness of a function instead by the $\ell^2$-norm of its discrete Laplacian (second derivative), which is given by $(\Delta f)(k) =(\nabla^2 f)(k)=f(k+2)-2f(k+1)+f(k)$. So we wish to find $u$ that uniformly minimizes $\| \Delta (f*u) \|_{\ell^2}$ over all choices of $f$ with fixed $\|f\|_{\ell^2}$.  This problem appears to be difficult in general (as discussed below), but we can obtain a solution if we additionally require $u$ to have nonnegative Fourier transform.
\begin{theorem}\label{thm:laplacian}
Let $u:\left\{-n, \dots, n\right\} \to \mathbb{R}$ be a symmetric function with normalization $\sum_{k=-n}^n {u(k)}=1$ and nonnegative Fourier transform. Then we have the inequality
$$\sup_{0 \neq f \in \ell^2(\mathbb{Z})}  \frac{\| \Delta (f*u)\|_{\ell^2(\mathbb{Z})}}{\|f\|_{\ell^2}} \geq \frac{4}{(n+1)^2},$$
with equality if and only if $u$ is the triangle function $u(k)=(n+1-|k|)/(n+1)^2$.
\end{theorem}
\noindent This inequality shows that convolving with any admissible kernel will, at least for one fuction $f \in \ell^2(\mathbb{Z})$, result in larger second derivatives than convolving with triangle function would; in this precise sense, the triangle function serves as the smoothest average for second derivatives.
We remark that we do not lose very much in either of these theorems by assuming that $u$ is symmetric: indeed, one can use the Triangle Inequality to show that, for each fixed $f$, the symmetrization $(u(k)+u(-k))/2$ performs at least as well as the original $u(k)$.  Our proofs will also show that these two theorems continue being sharp if we limit our attention to real-valued functions $f$ (as would be the case in many natural applications).
 

\subsection{Chebyshev polynomials as extremizers}

The $n$-th \emph{Chebyshev polynomial} (of the first kind) is the unique polynomial $T_n(x)$ of degree $n$ such that 
$$T_n(\cos(\xi))=\cos(n\xi).$$  Chebyshev polynomials are known to provide a solution to the problem of finding the monic polynomial of degree $n$ with the smallest possible deviation on an interval.

\begin{theorem}[Chebyshev \cite{cheb}]\label{thm:folklore}
For every monic polynomial $p(x)$ of degree $n$, we have the inequality
$$\max_{x \in [-1,1]} |p(x)| \geq 2^{1-n},$$
with equality if and only if $p(x)=2^{1-n}\cdot T_n(x)$.
\end{theorem}
\noindent We will require a variant of this problem in which one considers candidate polynomials with fixed sum of coefficients (instead of fixed leading coefficient). 
 Our main result says that this problem is solved by a simple modification of the Chebyshev polynomials.  We introduce the function
$$g_n(x)=\frac{1}{(n+1)^2} \cdot \frac{1-T_{n+1}(x)}{1-x}$$
and record the following observations: 
\begin{enumerate}
\item The function $g_n$ is a polynomial of degree $n$ since $$1-T_{n+1}(1)=1-T_{n+1}(\cos(0))=1-\cos(n \cdot 0)=0$$ and hence we can factor $(1-x)$ out of $1-T_{n+1}(x)$.
\item  The function $g_n(x)$ is nonnegative on $[-1,1]$ since $|T_{n+1}(x)| \leq 1$ on this interval.
\item We find that $g_n(1)=1$ by making the substitution $x=\cos(\xi)$ and computing
$$\lim_{x \to 1} \frac{1-T_{n+1}(x)}{1-x}=\lim_{\xi \to 0} \frac{1-\cos((n+1) \xi)}{1-\cos(\xi)}=(n+1)^2.$$
\end{enumerate}

\begin{theorem}\label{thm:main}
Let $p(x)$ be a polynomial of degree at most $n$ that is nonnegative on $[-1,1]$ and satisfies $p(1)=1$.  Then we have the inequality
$$\max_{x \in [-1,1]} (1-x)p(x) \geq \frac{2}{(n+1)^2},$$
with equality if and only if $p(x)=g_n(x)$.
\end{theorem}
\noindent 
The polynomial $g_{2n}(x)$  turns out to always be a perfect square, which allows us to solve a related problem.  We introduce the function
$$h_n(x)=\frac{1}{2n+1} \left( 1+2\sum_{k=1}^{n} T_k(x) \right).$$
We will see that $h_n(x)$ is a polynomial that satisfies $h_n(x)^2=g_{2n}(x)$.  We also note that $h_n(1)=1$ because $T_k(1)=1$.

\begin{theorem}\label{thm:pos-poly}
Let $p(x)$ be a polynomial of degree at most $n$ that satisfies $p(1)=1$.  Then we have the inequality
$$\max_{x \in [-1,1]} (1-x)p(x)^2 \geq \frac{2}{(2n+1)^2},$$
with equality if and only if $p(x)=h_n(x)$.
\end{theorem}

\noindent Theorems~\ref{thm:laplacian} and~\ref{thm:first-deriv} will turn out to be consequences of Theorems~\ref{thm:main} and~\ref{thm:pos-poly}, respectively.

\subsection{A continuous analogue}  While investigating the continuous version of the main Question, the second author established the following uncertainty principle.
\begin{theorem}[Steinerberger \cite{steini2}]\label{thm:stefan}
For every $\alpha > 0$ and $\beta > n/2$, there exists a constant $c_{\alpha, \beta,n} > 0$ such that for all functions $u \in L^1(\mathbb{R}^n)$, we have
$$  \| |\xi|^{\beta} \cdot \widehat{u}\|^{\alpha}_{L^{\infty}(\mathbb{R}^n)} \cdot \| |x|^{\alpha} \cdot u \|^{\beta}_{L^1(\mathbb{R}^n)} \geq c_{\alpha, \beta,n} \|u\|_{L^1(\mathbb{R}^n)}^{\alpha + \beta}.$$
\end{theorem}
\noindent It is natural to wonder about the existence, uniqueness and structure of extremizing functions. This determination is generally difficult, as is often the case for such sharp inequalities.  In the special case where $n=\beta=1$, Theorem~\ref{thm:stefan} simply says that
$$  \| |\xi|^{} \cdot \widehat{u}\|^{\alpha}_{L^{\infty}(\mathbb{R})} \cdot \| |x|^{\alpha} \cdot u \|^{}_{L^1(\mathbb{R})} \geq c_{\alpha}\|u\|_{L^1(\mathbb{R})}^{\alpha+1}.$$
It was then established in \cite{steini2} that, for $\alpha \in \left\{2,3,4,5,6\right\}$, the characteristic function $ u= \chi_{[-1,1]}$ is a local extremizer for this inequality in the class of compactly supported functions on $[-1,1]$ that are three-times continuously differentiable. The curious restriction of $\alpha$ to these indices is due to an algebraic step in the proof that relies on a certain sign pattern for $_2 F_1$.  This sign pattern seems easy to verify or falsify for any particular $\alpha \in \mathbb{N}$ (as was done for $\alpha \in \left\{2,3,4,5,6\right\}$), and it may well hold for all integers $\alpha \geq 2$. It is less clear how to establish the corresponding result for real $\alpha \geq 2$.\\

This local stability property of $\chi_{[-1,1]}$ can be interpreted as a continuous analog of Theorem~\ref{thm:first-deriv} (smoothness measured by the first derivative). Thus, our Theorem~\ref{thm:laplacian} (smoothness measured by the second derivative, with a restriction to functions with nonnegative Fourier transorm) corresponds to the $n=1$, $\beta=2$ case of Theorem~\ref{thm:stefan}, with the additional restriction that $\widehat{u}$ be nonnegative.  It is natural to ask whether or not  $u(x)=1-|x|$ is a local extremizer for the inequality
$$  \| |\xi|^{2} \cdot \widehat{u}\|^{\alpha}_{L^{\infty}(\mathbb{R})} \cdot \| |x|^{\alpha} \cdot u \|^{2}_{L^1(\mathbb{R})} \geq c_{\alpha} \|u\|_{L^1(\mathbb{R})}^{\alpha + 2}$$
in the class of $L^1$-functions with nonnegative Fourier transform.
 We specialize to the case $\alpha=2$, although (in analogy with the discussion above) our findings may be valid for a wide range of values of $\alpha$; a uniform treatment of all $\alpha$ seems to be more difficult.
\begin{theorem}\label{thm:triangle}
There exists a constant $c>0$ such that for all functions $u \in L^1(\mathbb{R})$, we have
$$ \|\widehat{u} \cdot |\xi|^2 \|^{2}_{L^{\infty}} \cdot \| u \cdot |x|^{2}\|^2_{L^1} \geq c \|u\|_{L^1}^{4}.$$
 Moreover, $u(x) = 1-|x|$ is a local extremizer in the class of all symmetric $C^3$-functions compactly supported on $[-1,1]$ with nonnegative Fourier transform.
\end{theorem}
\noindent More precisely, we will show that for each symmetric function $f:[-1,1] \to \mathbb{R}$  such that $\widehat{f}(\xi) \geq 0$ for all $\xi \in \mathbb{R}$ and $\widehat{f}(\xi)$ has sufficient decay, the functional
$$ J_f(\varepsilon) = \frac{\|\widehat{(u+\varepsilon f)} \cdot |\xi|^2 \|^{2}_{L^{\infty}} \cdot \| (u+\varepsilon f) \cdot |x|^{2}\|^2_{L^1}}{\|u+\varepsilon f\|_{L^1}^4}$$
satisfies $J'_f(0)> 0$. 
We do not know whether or not $1-|x|$ is actually a global extremizer among all functions with nonnegative Fourier transform (without any conditions on the support). The analogy with Theorem~\ref{thm:laplacian} suggests that it could be an optimizer among positive-definite functions with support on $[-1,1]$. The stability analysis in the proof of Theorem~\ref{thm:triangle} makes use of the following curious proposition that is also of interest in its own right.
\begin{proposition}\label{prop:hypergeo}
Let $f \in L^1[-1,1]$ satisfy $\widehat{f}(n) \geq 0$ for all $n \in \mathbb{Z} \setminus \left\{0\right\}$. Then
$$\sup_{n \in \mathbb{Z}}{ \widehat{f}\left(n+\frac{1}{2}\right)\cdot \left|n+\frac{1}{2}\right|^2} \geq  \frac{ 2}{ \pi^2} \int_{-1}^{1}{f(x)(1-3x^2) \,dx}.$$
with equality if and only if $f(x) = c(1-|x|)$ for some constant $c \geq 0$. \end{proposition}
\noindent Like its analog in \cite{steini2}, this amusing identity looks as if it may be a representative of a larger family of identities. It would be interesting to gain a better understanding of when such inequalities are possible.

\subsection{Open problems}
As suggested above, it is possible to measure smoothness with a differential operator other than the discrete derivative and Laplacian.  By taking a Fourier transform, one can reduce the problem of finding an optimal kernel $u$ to that of finding $u$ that minimizes
$$\max_{\xi \in \mathbb{T}} \left| r(e^{i \xi})  \right| \cdot \left| \sum_{k=-n}^{n}{u(k) e^{i k \xi}} \right|^2,$$
where $r$ is a polynomial that depends on the differential operator chosen.  This problem seems difficult in general, but it may be tractable for some particularly nice differential operators.  We also emphasize that the analog of Theorem~\ref{thm:laplacian} without the requirement of $u$ having nonnegative Fourier transform remains open.  Finally, it could be interesting to establish more robust and direct connections between the continuous and discrete versions of the questions raised in this paper.  This type of correspondence has recently proven fruitful in some areas of additive combinatorics (e.g.,
\cite{barnard, bern, cill, clon, noah, mad, mat}).

\section{Proofs}
In Section~\ref{sec:poly-proof}, we prove the main polynomial inequality Theorem~\ref{thm:main} and derive Theorem~\ref{thm:pos-poly} as a corollary.  In Section~\ref{sec:reduction}, we use Fourier analysis to reduce Theorems~\ref{thm:first-deriv} and~\ref{thm:laplacian} to Theorems~\ref{thm:pos-poly} and~\ref{thm:main}, respectively.  In Section~\ref{sec:triangle}, we prove Theorem~\ref{thm:triangle}, together with Proposition~\ref{prop:hypergeo}.

\subsection{Polynomial extremizers}\label{sec:poly-proof}
We begin by establishing Theorem~\ref{thm:main}.  The argument is a modification of the standard proof for Chebyshev's well-known Theorem~\ref{thm:folklore}.  Recall the definition of our (claimed) degree-$n$ extremal polynomial 
$$g_n(x)=\frac{1}{(n+1)^2} \cdot \frac{1-T_{n+1}(x)}{1-x}.$$

\begin{proof}[Proof of Theorem~\ref{thm:main}]
Consider the polynomial of degree $n+1$ given by
$$(1-x)g_n(x)=\frac{1}{(n+1)^2}(1-T_{n+1}(x)).$$
The relationship $T_{n+1}(\cos(\xi))=\cos((n+1)\xi)$ makes it clear that, for $x \in [-1,1]$, this polynomial assumes values between $0$ and $2/(n+1)^2$. Moreover, as $x$ decreases from $1$ to $-1$, it alternately assumes these two extremal values a total of $n+2$ times (including at $x=1$).
Now, let $p(x)$ be a polynomial of degree at most $n$ that is nonnegative on $[-1,1]$ and satisfies $p(1)=1$, and suppose that
$$\max_{x \in [-1,1]} (1-x)p(x) \leq \frac{2}{(n+1)^2}.$$
We will show that necessarily $p(x)=g_n(x)$.  Recall that  $g_n(1)=p(1)=1$.  In particular, the difference $p(x)-g_n(x)$ has a root at $x=1$, so we can write 
$$p(x)-g_n(x)=(1-x)q(x),$$ where $q(x)$ is a polynomial of degree at most $n-1$.  We will show that $q(x)\equiv 0$ uniformly.  Since $1-x$ is strictly positive on $[-1,1)$, we see that $(1-x)(p(x)-g_n(x))=(1-x)^2q(x)$ and $q(x)$ have the same sign everywhere on this interval.  For each $x^* \in [-1,1)$ satisfying $(1-x^*)g_n(x^*)=0$, the assumption on $p$ tells us that $(1-x^*)p(x^*) \geq 0$, whence we conclude that $q(x^*) \geq 0$.  By the same argument, we have that $q(x') \leq 0$ for each $x' \in [-1,1)$ satisfying $(1-x')g_n(x')=2/(n+1)^2$.  Since the $n+1$ values of $x^*$ and $x'$ interlace (as described above), we see that $q(x)$ has at least $n$ sign changes.  But the number of sign changes of a nonzero polynomial is at most its degree, so we conclude that $q(x)$ is the zero polynomial.  This concludes the proof.
\end{proof}

Theorem~\ref{thm:pos-poly} follows immediately from the claim that $h_n(x)^2=g_{2n}(x)$.  This relation is straightforward to check once we make the substitution $x=\cos(\xi)$ and recognize
$$h_n(\cos(\xi))=\frac{1}{2n+1} \cdot \frac{\sin((n+1/2)\xi)}{\sin(\xi/2)}$$
as the Dirichlet kernel $D_n(\xi)$.  The further computation
\begin{align*}
D_n(\xi)^2 &=\frac{1}{(2n+1)^2} \cdot \left( \frac{\sin((n+1/2)\xi)}{\sin(\xi/2)} \right)^2\\
 &=\frac{1}{(2n+1)^2} \cdot \frac{1-\cos^2((n+1/2)\xi)}{1-\cos^2(\xi/2)} \\
 &=\frac{1}{(2n+1)^2} \cdot \frac{1-\cos((2n+1)\xi)}{1-\cos(\xi)}
\end{align*}
shows that indeed $h_n(x)^2=g_{2n}(x)$. 

\begin{proof}[Proof of Theorem~\ref{thm:pos-poly}]
Note that $p(x)^2$ is a polynomial of degree $2n$ that is nonnegative on $[-1,1]$ and satisfies $p(1)^2=1^2=1$.  Then Theorem~\ref{thm:main} tells us that
$$\max_{x \in [-1,1]} (1-x)p(x)^2 \geq \frac{2}{(2n+1)^2},$$
with equality if and only if $p(x)^2=g_{2n}(x)$.  This equality condition establishes $p(x)$ up to a sign, the the assumption $p(1)=1$ tells us that we must choose $p(x)=h_n(x)$.
\end{proof}

\subsection{From discrete kernels to polynomial extremizers}\label{sec:reduction}

We begin by recalling a few facts from Fourier analysis. For any $f:\mathbb{Z} \to \mathbb{R}$, we consider its Fourier transform $\widehat{f}: \mathbb{T} \to \mathbb{C}$ given by
$$ \widehat{f}(\xi) = \sum_{k \in \mathbb{Z}} f(k) e^{- i \xi k}.$$
We will sometimes write $\widehat{f}=\mathcal{F}(f)$ for readability.  We recall the Convolution Theorem
$$ \widehat{ f * g} = \widehat{f} \cdot \widehat{g}$$
and the Plancherel Identity
$$ \sum_{k \in \mathbb{Z}} f(k)\overline{g(k)} = \frac{1}{2\pi} \int_{\mathbb{T}} \widehat{f}(\xi)\overline{\widehat{g}(\xi)} \,d\xi.$$
We also mention that the Fourier transform of a shifted function $f(k-m)$ is given by
\begin{align*}\widehat{f(k-m)} &=  \sum_{k \in \mathbb{Z}} f(k-m) e^{- i \xi k}=  \sum_{k \in \mathbb{Z}} f(k) e^{- i \xi (k+m)}\\
&= e^{-i \xi m} \sum_{k \in \mathbb{Z}} f(k) e^{- i \xi k} = e^{-i \xi m} \widehat{f}(\xi).
\end{align*}
For instance, we immediately have
$$\mathcal{F}(\nabla f)=(e^{i \xi}-1)\widehat{f}$$
and
$$\mathcal{F}(\Delta f)=(e^{2i \xi}-2e^{i \xi}+1)\widehat{f}=(e^{i \xi}-1)^2\widehat{f}.$$
We can now proceed with the proofs of Theorems~\ref{thm:first-deriv} and~\ref{thm:laplacian}.

\begin{proof}[Proof of Theorem~\ref{thm:first-deriv}]
We aim to understand the behavior of 
\begin{align*}
\| \nabla (f*u)\|_{\ell^2}^2 = \sum_{k \in \mathbb{Z}} |(\nabla (f*u))(k)|^2.
\end{align*}
Applying the Plancherel Identity leads to the estimate
\begin{align*}
\sum_{k \in \mathbb{Z}} |(\nabla (f*u))(k))|^2 &= \frac{1}{2\pi} \int_{\mathbb{T}}  |e^{i \xi} - 1|^2 |\widehat{f}(\xi)|^2  |\widehat{u}(\xi)|^2 \,d\xi\\
&\leq \| |e^{i \xi} - 1|^2 |\widehat{u}(\xi)|^2 \|_{L^{\infty}(\mathbb{T})}  \cdot  \frac{1}{2\pi} \int_{\mathbb{T}} |\widehat{f}(\xi)|^2 \,d\xi \\
&= \| |e^{i \xi} - 1|^2 |\widehat{u}(\xi)|^2 \|_{L^{\infty}(\mathbb{T})}   \cdot \sum_{k \in \mathbb{Z}}{|f(k)|^2},
\end{align*}
and taking square roots gives
$$\| \nabla (f*u)\|_{\ell^2(\mathbb{Z})} \leq \| (e^{i \xi} - 1)\widehat{u}(\xi)\|_{L^{\infty}(\mathbb{T})} \cdot \| f\|_{\ell^2}.$$
We claim that, for each choice of $u$, the constant $\|(e^{i \xi} - 1)\widehat{u}(\xi)\|_{L^{\infty}(\mathbb{T})}$ on the right-hand side is in fact optimal: the (only) inequality in these calculations can be made arbitrarily close to equality by taking $\widehat{f}$ to have mass concentrated at a value of $\xi \in \mathbb{T}$ where $(e^{i \xi} - 1)\widehat{u}(\xi)$ achieves its maximum magnitude. (Since our extremizing choices for $u$ will turn out to be real, the real part of such an approximating function $f$ shows that one cannot hope for a better constant by restricting to purely real functions.)  Thus, we conclude that
$$\sup_{0 \neq f \in \ell^2(\mathbb{Z})} \frac{ \| \nabla (f*u)\|_{\ell^2(\mathbb{Z})}}{\| f\|_{\ell^2}} =\| (e^{i \xi} - 1)\widehat{u}(\xi)\|_{L^{\infty}(\mathbb{T})},$$
so our problem is reduced to minimizing the quantity $$M(u)=\|(e^{i \xi} - 1)\widehat{u}(\xi)\|_{L^{\infty}(\mathbb{T})},$$
among all symmetric functions $u: \left\{-n, \dots, n\right\} \to \mathbb{R}$ with normalization $ \sum_{k \in \mathbb{N}}{ u(k)} = 1$.\\

Note that
$$M(u)^2=2\max_{0 \leq \xi \leq 2\pi} (1-\cos{\xi}) \left|  \sum_{k \in \mathbb{Z}} u(k) e^{- i \xi k} \right|^2.$$
Since $u$ is symmetric and real-valued, we can write
$$\left| \sum_{k\in \mathbb{Z}} u(k) e^{- i \xi k} \right|^2 =\left( u(0) + \sum_{k=1}^{n}{2u(k) \cos{(k \xi)}} \right)^2.$$
Using Chebyshev polynomials to expand the cosines and then making the substitution $x=\cos \xi$, we get
$$M(u)^2=2\max_{-1 \leq x \leq 1}(1-x)p_u(x)^2,$$
where
$$p_u(x)=u(0)+\sum_{k=1}^n 2u(k) T_k(x)$$
is a real-valued polynomial of degree at most $n$. Note that, since each Chebyshev polynomial satisfies $T_k(1)=1$, we have
$$p_u(1)=u(0)+\sum_{k=1}^n 2u(k) T_k(1)=\sum_{k=-n}^n u(k) =1.$$
Theorem~\ref{thm:pos-poly} tells us that
$$M(u)^2 \geq \frac{4}{(2n+1)^2}, \quad \text{i.e.,} \quad M(u) \geq \frac{2}{(2n+1)},$$
with equality only for $p_u(x)=h_n(x)$.  Finally, it is immediate from the definition of $h_n(x)$ that $p_u(x)=h_n(x)$ corresponds to the choice $u(k)=1/(2n+1)$.
\end{proof}

We prove Theorem~\ref{thm:laplacian} in much the same way.

\begin{proof}[Proof of Theorem~\ref{thm:laplacian}]
Replacing $\Delta$ with $\nabla$ in the argument from the previous proof shows that
$$\sup_{0 \neq f \in \ell^2(\mathbb{Z})} \frac{\| \Delta (f*u) \|_{\ell^2}}{ \|f\|_{\ell^2}}= \| (e^{i \xi} - 1)^2 \widehat{u}(\xi) \|_{L^{\infty}(\mathbb{T})},$$
so our problem is reduced to minimizing the quantity
$$L(u)=\| (e^{i \xi} - 1)^2 \widehat{u}(\xi) \|_{L^{\infty}(\mathbb{T})},$$
among all symmetric functions $u: \left\{-n, \dots, n\right\} \to \mathbb{R}$ with the normalization $ \sum_{k \in \mathbb{N}}{ u(k)} = 1$ and the additional hypothesis that $\widehat{u}$ is nonnegative.
Expanding $\widehat{u}$ using Chebyshev polynomials and substituting $x=\cos \xi$ as before, we find that
$$L(u)=2 \max_{-1\leq x \leq 1}(1-x)p_u(x),$$
where
$$p_u(x)=u(0)+\sum_{k=1}^n 2u(k) T_k(x)$$
is a real-valued polynomial of degree at most $n$.  As before, we have $p_u(1)=1$.  Note also that $p_u(x) \geq 0$ on $[-1,1]$ by the assumption on $\widehat{u}$.  Theorem~\ref{thm:main} tells us that
$$L(u) \geq \frac{4}{(n+1)^2},$$
with equality only for $p_u(x)=g_n(x)$.
It remains to show that $p_u(x)=g_n(x)$ corresponds to the choice $u(k)=(n+1-|k|)/(n+1)^2$.  We recognize
$$\widehat{u}(\xi)=u(0)+\sum_{k=1}^n 2u(k) \cos(k \xi)=\frac{1}{(n+1)^2} \cdot \frac{1-\cos((n+1)\xi)}{1-\cos(\xi)}$$
as a normalization of the Fej\'{e}r kernel $F_n(\xi)$.  (Another way to see this is to note that the discrete triangle function is the autoconvolution of the constant function and that the Fej\'{e}r kernel is a normalization of the square of the Dirichlet kernel.)  We immediately see that $p_u(x)=g_n(x)$ in this case, as desired.
\end{proof}

\subsection{The continuous triangle function kernel}\label{sec:triangle}
We first show how to reduce Theorem~\ref{thm:triangle} to Proposition~\ref{prop:hypergeo}.  The idea is to take $u_0(x) = \max\{1-|x|,0\}$ as our candidate kernel and consider the effect of a slight perturbation.  Our main tool is Taylor series analysis (in the size of the perturbation).  Given a function $f: \mathbb{R} \to \mathbb{C}$, we will work with its Fourier transform $\widehat{f}: \mathbb{R} \to \mathbb{C}$ as given by
$$\widehat{f}(\xi)=\int_{\mathbb{R}}f(x) e^{-2 \pi i\xi x} \,dx.$$
A crucial ingredient of the proof is to use the compact support of $f$ on $[-1,1]$. This allows us to also describe $f$ via Fourier coefficients (in this case the value of the Fourier transform of $f$ evaluated at $\mathbb{Z}/2$).

\begin{proof}[Proof of Theorem~\ref{thm:triangle}]
Fix a function $f:[-1,1] \to \mathbb{R}$ that is three times continuously differentiable and has strictly positive Fourier transform, and consider the functional
$$ J_f(\varepsilon) = \frac{\|\widehat{(u_0+\varepsilon f)} \cdot |\xi|^2 \|^{2}_{L^{\infty}} \cdot \| (u_0+\varepsilon f) \cdot |x|^{2}\|^2_{L^1}}{\|u_0+\varepsilon f\|_{L^1}^4}.$$ 
We will expand this functional up to first order and show that 
$$ J_f(\varepsilon) = J_f(0) + c_f \varepsilon + o(\varepsilon),$$
where $c_f> 0$ is a constant depending only on $f$.  We will examine the behavior of the terms in $J_f(\varepsilon)$ one at a time.\\

First, consider
$ I(\varepsilon)= \| (\widehat{u_0} + \varepsilon \widehat{f}) \cdot |\xi|^2 \|^{}_{L^{\infty}}$
as $\varepsilon \to 0$.  Computing the Fourier transform
$$ \widehat{u_0}(\xi) = \frac{\sin{(\pi \xi)^2}}{(\pi \xi)^2},$$
we note that $\widehat{u_0}(\xi) \cdot |\xi|^2=\sin(\pi \xi)^2/\pi^2$ oscillates between $0$ and $1/\pi^2$ and attains the latter value precisely when $\xi$ is a half-integer.  Our goal is to show that
\begin{equation}\label{eq:perturb}
I(\varepsilon) = \frac{1}{\pi^2} + \varepsilon  \sup_{n \in \mathbb{Z}}{ \widehat{f}\left(n+\frac{1}{2}\right)\cdot \left|n+\frac{1}{2}\right|^2} + o(\varepsilon),
\end{equation}
and our strategy involves the following three steps:
\begin{enumerate}[(i)]
\item For each $\varepsilon>0$, there is some $\xi$ such that $|(\widehat{u_0} + \varepsilon \widehat{f}) \cdot |\xi|^2|=I(\varepsilon)$, i.e., the maximum magnitude is actually attained.
\item This maximum is attained only for $\xi$ very close to a half-integer.  More precisely, every $\xi$ with $|(\widehat{u_0} + \varepsilon \widehat{f}) \cdot |\xi|^2|=I(\varepsilon)$ must lie at a distance of at most $\mathcal{O}(\sqrt{\varepsilon})$ from some half-integer.
\item Finally, outside of some trivial cases, taking $\varepsilon$ sufficiently small guarantees that the maximum is attained only near half-integers of a uniformly bounded size.
\end{enumerate}

We see that (i) is trivially satisfied if $I(\varepsilon)=\|\widehat{u_0} \cdot |\xi|^2\|_{L^{\infty}}=1/\pi^2$: indeed, the maximum magnitude is achieved whenever $\xi$ is a half-integer (because $\widehat{f}$ is nonnegative).  Suppose now that $I(\varepsilon)>\|\widehat{u_0} \cdot |\xi|^2\|_{L^{\infty}}$ strictly.  Since $f$ is $C^3$, we know that $\widehat{f}$ decays at least as quickly as $|\xi|^{-3}$.  In particular, $\widehat{f}(\xi)\cdot |\xi|^2=\mathcal{O}(1/|\xi|)$, and this quantity is smaller than $(I(\varepsilon)-\|\widehat{u_0} \cdot |\xi|^2\|_{L^{\infty}})/2$ for $|\xi|$ sufficiently large.  Using the Triangle Inequality, we conclude that $(\widehat{u_0}+\varepsilon\widehat{f})\cdot |\xi|^2$ approaches its supremum only within some bounded closed interval and hence that the supremum is actually attained (by compactness), as desired.\\

Next, we show that $(\widehat{u_0} + \varepsilon \widehat{f}) \cdot |\xi|^2$ can attain its maximum magnitude only close to where $\widehat{u_0} \cdot |\xi|^2=\sin(\pi \xi)^2/\pi^2$ attains its maximum magnitude, i.e., at half-integers.  Set $ \xi = n + 1/2+ \delta$, with $n \in \mathbb{Z}$ and $|\delta|\leq 1/2$.  The standard bound $\sin(\pi \xi)^2 \leq 1-4\delta^2$ gives the inequality
$$ (\widehat{u_0}(\xi) + \varepsilon \widehat{f}(\xi)) \cdot |\xi|^2 \leq \frac{1}{\pi^2} - \frac{4\delta^2}{\pi^2} + \varepsilon |\xi|^2 \cdot \widehat{f}(\xi).$$
Since $\widehat{f}(\xi)\cdot |\xi|^2=\mathcal{O}(1/|\xi|)$, we have a uniform bound $\| \widehat{f}(\xi) \cdot |\xi|^2 \|_{L^{\infty}}=K<\infty$.  Putting these two facts together shows that
$$(\widehat{u_0}(\xi) + \varepsilon \widehat{f}(\xi)) \cdot |\xi|^2 < \frac{1}{\pi^2} \leq I(\varepsilon)$$
unless
$$|\delta| \leq \frac{\pi \sqrt{K \varepsilon}}{2},$$
which establishes (ii).  Henceforth, we restrict our attention to $\delta$ in this range.  We can say even more: since $f$ is compactly supported, $\widehat{f}$ has derivatives of all order, and each such derivative is uniformly bounded.  Expanding $\varepsilon \widehat{f}$ around $\xi=n+1/2$ gives
\begin{equation}\label{eq:bound-on-f}
\varepsilon \widehat{f}(\xi)\leq \varepsilon \widehat{f}\left(n+\frac{1}{2}\right)+\varepsilon |\delta| \cdot \left\|\frac{d}{d\xi}  \widehat{f} \right\|_{L^{\infty}}=\varepsilon \widehat{f}\left(n+\frac{1}{2}\right)+\mathcal{O}(\varepsilon^{3/2}),
\end{equation}
where we used the estimate $\delta =\mathcal{O}(\sqrt{\varepsilon})$.
We now distinguish two possibilities for the values of $\widehat{f}$ at the points $n+1/2$.
First, suppose $\widehat{f}(n+1/2)=0$ for all $n \in \mathbb{Z}$.  Using the fact that all derivatives decay at least as quickly as $|\xi|^{-3}$ (and, in particular, faster than $|\xi|^{-2}$), we see that
$$ \| (\widehat{u_0} + \varepsilon \widehat{f}) \cdot |\xi|^2 \|^{}_{L^{\infty}} = \frac{1}{\pi^2} + o(\varepsilon),$$
which certainly satisfies Equation~\eqref{eq:perturb}.
Henceforth, we restrict our attention to the case where $\widehat{f}(n+1/2)$ is not uniformly $0$ for $n \in \mathbb{Z}$.  Choose some $m$ such that $\widehat{f}(m+1/2)>0$; it follows that $I(\varepsilon)\geq 1/\pi^2+(\widehat{f}(m+1/2))\varepsilon$ grows at least linearly.  Following the discussion in (i), we see that $(\widehat{u_0}+\varepsilon\widehat{f})\cdot |\xi|^2<I(\varepsilon)$ outside of some bounded interval that is independent of the choice of $\varepsilon$.  This establishes (iii).\\

The ``easy'' half of the inequality \eqref{eq:perturb} is
$$ I(\varepsilon) \geq \frac{1}{\pi^2} + \varepsilon  \sup_{n \in \mathbb{Z}}{ \widehat{f}\left(n+\frac{1}{2}\right)\cdot \left|n+\frac{1}{2}\right|^2}.$$
Combining the observation from (iii) with Equation~\eqref{eq:bound-on-f} gives that
$$ I(\varepsilon) \leq \frac{1}{\pi^2} + \varepsilon  \sup_{n \in \mathbb{Z}}{ \widehat{f}\left(n+\frac{1}{2}\right)\cdot \left|n+\frac{1}{2}\right|^2} + \mathcal{O}(\varepsilon^{3/2}),$$
whence we deduce \eqref{eq:perturb}.  Squaring gives $$I(\varepsilon)^2=\frac{1}{\pi^4}+\frac{2\varepsilon}{\pi^2}\sup_{n \in \mathbb{Z}}{ \widehat{f}\left(n+\frac{1}{2}\right)\cdot \left|n+\frac{1}{2}\right|^2}+o(\varepsilon).$$
The other two terms in the statement of the theorem are easy to linearize.  Since $u_0$ is nonnegative, we have
\begin{align*}
 \| (u_0+\varepsilon f) \cdot |x|^{2}\|_{L^1}  &= \int_{-1}^{1}{(1-|x|) |x|^{2} \,dx} + \varepsilon  \int_{-1}^{1}{f(x) \cdot |x|^{2} \,dx} + \mathcal{O}(\varepsilon^2) \\
 &=\frac{1}{6} + \varepsilon  \int_{-1}^{1}{f(x) \cdot |x|^{2} \,dx} + \mathcal{O}(\varepsilon^2) 
 \end{align*}
and thus
$$ \| (u_0+\varepsilon f) \cdot |x|^{2}\|^2_{L^1}  = \frac{1}{36} + \frac{\varepsilon}{3}  \int_{-1}^{1}{f(x) \cdot |x|^{2} \,dx} + \mathcal{O}(\varepsilon^2).$$
As for the last term, it is easy to see that
$$  \|u_0+ \varepsilon f\|_{L^1}^{4} = \left(1 + \varepsilon \int_{-1}^{1}{f(x) dx} \right)^{4} = 1 + 4 \varepsilon \int_{-1}^{1}{f(x) dx} + \mathcal{O}(\varepsilon^2).$$ 
Collecting all of the terms, we see that
$$J_f(\varepsilon)=\frac{1}{36\pi^4}+c_f \varepsilon+o(\varepsilon),$$
where
$$c_f=\frac{1}{3 \pi^4} \int_{-1}^{1} f(x) \cdot |x|^2 \,dx+\frac{1}{18 \pi^2}\sup_{n \in \mathbb{Z}}{ \widehat{f}\left(n+\frac{1}{2}\right)\cdot \left|n+\frac{1}{2}\right|^2}-\frac{1}{9 \pi^4} \int_{-1}^1 f(x) \, dx.$$
Regrouping, we find the statement $c_f>0$ to be equivalent to the inequality
$$\sup_{n \in \mathbb{Z}}{ \widehat{f}\left(n+\frac{1}{2}\right)\cdot \left|n+\frac{1}{2}\right|^2}> \frac{ 2}{ \pi^2} \int_{-1}^{1}{f(x)(1-3x^2) \,dx},$$
which is the content of Proposition~\ref{prop:hypergeo}.
\end{proof}

Finally, we complete the argument by proving Proposition~\ref{prop:hypergeo}.

\begin{proof}[Proof of Proposition~\ref{prop:hypergeo}]
We can use the Plancherel Identity to obtain
$$ \int_{-1}^{1}{f(x)(1-3x^2) \,dx} = \frac{1}{2}  \sum_{j \in \mathbb{Z}/2}{ a_j \widehat{f}(j)},$$
where
$$ a_j = \int_{-1}^{1}{(1-3x^2)\cos{(2\pi j x)} \,dx}$$
and $\mathbb{Z}/2$ denotes the set of integers and half integers.  We observe first that $a_0 = 0$.  For any nonzero integer $k$, we have
$$ a_k = - \frac{3}{k^2 \pi^2} < 0.$$
Finally, for any integer $k$, we have
$$ a_{k+1/2} = \frac{12}{(2k+1)^2 \pi^2} > 0.$$
Let us introduce the parameter
$$ \gamma =  \sup_{n \in \mathbb{Z}}{ \widehat{f}\left(n+\frac{1}{2}\right)\cdot \left|n+\frac{1}{2}\right|^2},$$
so that
$$ \widehat{f}\left(n+\frac{1}{2}\right) \leq \frac{\gamma}{|n+\frac12|^2}.$$
Since $\widehat{f}(j) \geq 0$ for all $j$, we can argue that
\begin{align*}
 \int_{-1}^{1}{f(x)(1-3x^2) \,dx} &=\frac{1}{2} \left[ \sum_{k \in \mathbb{Z}}{ a_k \widehat{f}(k)} + \sum_{k \in \mathbb{Z}}{ a_{k+1/2} \widehat{f}\left(k+ \frac{1}{2}\right)} \right]\\
 &\leq  \frac{1}{2}  \sum_{k \in \mathbb{Z}}{ a_{k+1/2} \widehat{f}\left(k+ \frac{1}{2}\right)}\\
 &\leq \frac{\gamma}{2} \sum_{k \in \mathbb{Z} }{ \frac{12}{(2k+1)^2 \pi^2} \cdot \frac{1}{(k+\frac12)^2}} \\
 &=\frac{24 \gamma}{\pi^2} \sum_{k \in \mathbb{Z} }{ \frac{1}{(2k+1)^4 }} = \gamma \frac{\pi^2}{2}.
\end{align*}
It remains to characterize cases of equality.  Suppose that both of the inequalities in the above calculation are equalities.  From the first, we see that $\widehat{f}(n)=0$ for all $n \in \mathbb{Z} \setminus\{0\}$.  From the second, we see that
$$ \widehat{f}\left(n+\frac{1}{2}\right) = \frac{\gamma}{|n+\frac12|^2}$$
for all $n \in \mathbb{Z}$.  Note that we have not yet said anything about the ``constant value'' $\widehat{f}(0)$ (essentially because $a_0=0$ obscures any such information).  Since our function $f$ is compactly supported on $[-1,1]$ by assumption, it is completely determined by the values of its Fourier transform at the integers and half-integers.  Since the value of $\widehat{f}(0)$ just determines the mean value of $f$, our information suffices to determine $f$ up to an additive shift.  So we see that $$f(x)=c(1-|x|)\chi_{[-1,1]}(x)+d$$ for some $c,d \in \mathbb{R}$ (where $c$ depends on $\gamma$, and $d$ depends on both $\gamma$ and $\widehat{f}(0)$).  Since $f$ is supported on $[-1,1]$ by assumption, we must have $d=0$.  Finally, we note that $c \geq 0$ because of the condition that $\widehat{f}$ is nonnegative. 
\end{proof}

\section*{Acknowledgements}
S.S. is supported by the NSF (grant DMS-1763179) and the Alfred P. Sloan Foundation.

\end{document}